\documentclass[12pt]{amsart}
\usepackage{amsfonts}

\newtheorem{theorem}{Theorem}

\theoremstyle{definition}

\theoremstyle{remark}

\let \f=\varphi

\begin{document}
\title[Multidimensional Hausdorff operators]
{Boundedness of multidimensional Hausdorff\\ operators on $L^1$
and $H^1$ spaces}

\author{Elijah Liflyand}
\address{Department of Mathematics, Bar-Ilan University, Ramat-Gan 52900, Israel}
\email{liflyand@yahoo.com, liflyand@math.biu.ac.il}
\urladdr{www.math.biu.ac.il/~liflyand}

\begin{abstract}
For a wide family of multivariate Hausdorff operators, a new
stronger condition for the boundedness of an operator from this
family on the real Hardy space  $H^1$ by means of atomic
decomposition.
\end{abstract}

\maketitle

2000 {\it Mathematics Subject Classification.} Primary 47B38,
42B10; Secondary 46E30.

\quad

Key words and phrases: Hausdorff operator, real Hardy space,
atomic decomposition, eigenvalues.

\quad

\section{Introduction}

In the one-dimensional case Hausdorff operators on the real line
were introduced in \cite{Ge} and studied on the Hardy space in
\cite{LM1}. As in \cite{BM2}, we define a multidimensional
Hausdorff type operator by

\begin{eqnarray}\label{haop}
({\mathcal H}f)(x)=({\mathcal H}_{\Phi}f)(x)=({\mathcal
H}_{\Phi,A}f)(x) =\int_{{\mathbb R}^n}\Phi(u)f\big(xA(u)\big)\,du,
\end{eqnarray}
where $A=A(u)=(a_{ij})_{i,j=1}^n=\big(a_{ij}(u)\big)_{i,j=1}^n$ is
the $n\times n$ matrix with the entries $a_{ij}(u)$ being
measurable functions of $u$. This matrix may be singular at most
on a set of measure zero; $xA(u)$ is the row $n$-vector obtained
by multiplying the row $n$-vector $x$ by the matrix $A.$ Of
course, $xA$ can be written as $A^T x^T,$ where both matrix and
vector are transposed, the latter to a column vector.

We are going to prove sufficient conditions, in terms of $\Phi$
and $A,$ for the boundedness of the whole range of Hausdorff type
operators (\ref{haop}) in $H^1({\mathbb R}^n).$

Before proving the result we give natural assumptions on $\Phi$
and $A,$ which provide the boundedness of the Hausdorff operator
in $L^1({\mathbb R}^n).$

Let the following condition be satisfied:

\begin{eqnarray}\label{l1}\|\Phi\|_{L_{A}}=\int_{{\mathbb R}^n}|\Phi(u)|\,|\det
A^{-1}(u)|\,du<\infty,                            \end{eqnarray}
or $\f(u)=\Phi(u)\det A^{-1}(u)\in L^1({\mathbb R}^n).$

Among the other basic properties of Hausdorff operators, one may
find in \cite{BM2}, in slightly different terms, that the operator
${\mathcal H}f$ is bounded taking $L^1$ into $L^1,$ with

\begin{eqnarray*}\|{\mathcal
H}f\|_{L^1}\le\|\Phi\|_{L_A}\|f\|_{L^1}.\end{eqnarray*}

It was proved in \cite{Mor} that the same condition provides the
boundedness of Hausdorff type operators on $H^1(\mathbb R^n)$ for
diagonal matrices $A$ with all diagonal entries equal to one
another.

We denote

\begin{eqnarray*}\|B\|=\|B(u)\|=\max_j(|b_{1j}(u)|+...+|b_{nj}(u)|),\end{eqnarray*}
where $b_{nj}$ are the entries of the matrix $B,$ to be the
operator $\ell$-norm. We will say that $\Phi\in L^*_{B}$ if
\begin{eqnarray*}
\|\Phi\|_{L^*_{B}}=\int_{{\mathbb R}^n}|\Phi(u)|\,\|B(u)\|^n
du<\infty.\end{eqnarray*}

The following result was proved by Lerner and Liflyand \cite{LL}
for the boundedness of Hausdorff type operators in $H^1(\mathbb
R^n)$ for general matrices $A.$ The proof used duality argument.

\begin{theorem}\label{hh1}

The Hausdorff operator ${\mathcal H}f$ is bounded on the real
Hardy space $H^1(\mathbb R^n)$ provided $\Phi\in L_{A^{-1}}^*,$
and

\begin{eqnarray}\label{ed} \|{\mathcal H}f\|_{H^1(\mathbb R^n)}\le
\|\Phi\|_{L_{A^{-1}}^*}\|f\|_{H^1(\mathbb R^n)}.\end{eqnarray}

\end{theorem}

The difference in conditions $\Phi\in L_{A^{-1}}$ and $\Phi\in
L_{A^{-1}}^*$ seemed to be quite natural. In \cite{LL} and then in
\cite{Lif} the problem of the sharpness of Theorem \ref{hh1} was
posed. We will prove that a weaker condition provides the
boundedness of Hausdorff type operators on $H^1(\mathbb R^n).$ The
proof will be based on atomic decomposition of $H^1(\mathbb R^n).$

In what follows $a\ll b$ means that $a\le Cb$ for some absolute
constant $C$ but we are not interested in explicit indication of
this constant.

\section{Main result and proof}

Let $||B||_2=\max_{|x|=1}|Bx^T|,$ where $|\cdot|$ denotes the
Euclidean norm. It is known (see, e.g., \cite[Ch.5, 5.6.35]{HJ})
that this norm does not exceed any other matrix norm. We will say
that $\Phi\in L^2_{B}$ if

\begin{eqnarray*}
\|\Phi\|_{L^2_{B}}=\int_{{\mathbb R}^n}|\Phi(u)|\,\|B(u)\|_2^n
du<\infty.\end{eqnarray*}The following result is true.

\begin{theorem}\label{hh2}

The Hausdorff operator ${\mathcal H}f$ is bounded on the real
Hardy space $H^1(\mathbb R^n)$ provided $\Phi\in L^2_{A^{-1}},$
and

\begin{eqnarray}\label{ed2} \|{\mathcal H}f\|_{H^1(\mathbb R^n)}\ll
\|\Phi\|_{L^2_{A^{-1}}}\|f\|_{H^1(\mathbb R^n)}.\end{eqnarray}

\end{theorem}

\begin{proof}

Let $a(x)$ denote an atom (a $(1,\infty,0)$-atom), a function
satisfying the following conditions:

\begin{eqnarray}\label{c1} \mbox{\rm supp}\,a\subset
B(x_0,r);\end{eqnarray}

\begin{eqnarray}\label{c2} ||a||_\infty\le\frac{1}{
|B(x_0,r)|};\end{eqnarray}

\begin{eqnarray}\label{c3} \int_{\mathbb R^n} a(x)\,dx=0.
\end{eqnarray}

It is well known that

\begin{eqnarray}\label{ha} ||f||_{H^1}\sim\inf\{\sum\limits_k |c_k|:
f(x)=\sum\limits_k c_k a_k(x)\},                \end{eqnarray}
where $a_k$ are atoms.

The other value to which $||f||_{H^1}$ is equivalent is

\begin{eqnarray}\label{hr} \sum\limits_{p=0}^n \int_{\mathbb R^n}
|R_pf(x)|\,dx,\end{eqnarray} where $R_0f\equiv f$ and $R_p$ are
$n$ Riesz transforms (see, e.g., \cite{We}).

We now have

\begin{eqnarray*} &||{\mathcal H}f||_{H^1}=||\int_{{\mathbb R}^n}
\Phi(u)f(\cdot A(u))\,du||_{H^1}\\
&\ll \sum\limits_{p=0}^n \int_{\mathbb R^n} |R_p{\mathcal H}
f(x)|\,dx\le \int_{\mathbb R^n} |\Phi(u)| \sum\limits_{p=0}^n
||R_pf(\cdot A(u))||_{L^1} du\\
&\ll \int_{\mathbb R^n} |\Phi(u)|\, ||f(\cdot A(u)||_{H^1}du.
\end{eqnarray*}

We wish to estimate the right-hand side from above by using
(\ref{ha}). Let

\begin{eqnarray*} f(x A(u))=\sum\limits_k c_k
a_k(xA(u)).                                      \end{eqnarray*}
We will show that multiplying $a_k(xA(u))$ by a constant depending
on $u$ (actually on $A(u)$) we get an atomic decomposition of $f$
itself, with no composition in the argument. Since we analyze all
such decompositions for $f,$ the upper bound will be $||f||_{H^1}$
times the mentioned constant, which completes the proof.

Thus, let us figure out when, or under which transformation
$a_k(xA(u))$ becomes an atom. We have

\begin{eqnarray*} \int_{\mathbb R^n} a_k(xA(u))\,dx=
\int_{a_k(xA(u))\neq0} a_k(xA(u))\,dx,           \end{eqnarray*}
and under substitution $xA(u)=v$ the integral becomes
$\int_{\mathbb R^n} a_k(v)\,dv$ times a Jacobian depending only on
$u.$ This integral vanishes because of (\ref{c3}).

The support of $a_k(xA(u))$ is $<xA,xA>\le r^2,$ an ellipsoid. To
use known results, let us represent it in the transposed form
$<A^T x^T,A^T x^T>\le r^2.$

Let us solve the following extremal problem. We are looking for
the $\min$ of the quadratic form $<Bx^T,Bx^T>$, where $B$ is a
non-singular $n\times n$ real-valued matrix - we denote the linear
transformation and its matrix with the same symbol - on the unit
sphere $<x^T,x^T>=1.$

Denoting by $B^*$ the adjoint to $B,$ we arrive to the equivalent
problem for $<B^* Bx^T,x^T>.$ Since $(B^*)^*=B$, the operator $B^*
B$ is self-adjoint:

\begin{eqnarray*}<B^* Bx,y>=<Bx,By>=<x,B^* By>,\end{eqnarray*}
and thus positive definite. Since $B$ is non-singular, the same
$B^* B$ is.

If a transformation is positive definite, all its eigenvalues are
non-negative; if it is also non-singular all the eigenvalues are
strictly positive. Define the minimal eigenvalue of $B^* B$ by
$l_1.$

By Theorem 1 from Ch.II, \S17 of \cite{Gel}, for self-adjoint $C$
the form $<Cx,x>$ on the unit sphere attains its minimum equal to
the least eigenvalue of $C.$

Hence the solution of the initial problem is just $l_1.$ We
mention also that the matrix of the adjoint real transformation is
the transposed initial. Therefore, the desired minimum is the
least eigenvalue $l_1$ of $B^T B.$

It follows from this by taking $B=A^T$

\begin{eqnarray*} l_1 <x^T,x^T>\le<A^T x^T,A^T x^T>\le
r^2,                                            \end{eqnarray*}
and every point of the ellipsoid $<A^T x^T,A^T x^T>\le r^2$ lies
in the ball $<x^T,x^T>\le r^2/l_1,$ where $l^1$ is the minimal
eigenvalue of the matrix $A^T A.$

It remains to check the $\infty$ norm. Instead of the measure of
the ball in (\ref{c2}) we must have the measure of the ball of
radius $r/\sqrt{l_1}.$ This is achieved by multiplying
$a_k(xA(u))$ by $l_1^{n/2}$ and hence $l_1^{n/2}a_k(xA(u))$ is an
atom. Correspondingly,

\begin{eqnarray*} ||f(\cdot A(u))||_{H^1}\ll l_1^{-n/2}||f||_{H^1},
\end{eqnarray*}
and finally ${\mathcal H} f$ belongs to $H^1$ provided

\begin{eqnarray}\label{ed0} \int_{\mathbb R^n} |\Phi(u)|\,
l_1^{-n/2}(u)\,du<\infty.   \end{eqnarray}

Further, $l_1^{-n/2}=L_n^{n/2},$ where $L_n$ is the maximal
eigenvalue of the matrix $(A^TA)^{-1}=A^{-1}(A^T)^{-1}.$ But it is
known that such $L_n$ is equal to the spectral radius of the
corresponding matrix $A^{-1}(A^T)^{-1}$ and, in turn, to
$||A^{-1}||_2^2.$  Replacing $l_1^{-n/2}(u)$ in (\ref{ed0}) with
the obtained bound completes the proof. \hfill\end{proof}

As is mentioned above, the obtained condition (\ref{ed2}) is
weaker that (\ref{ed}) but of course still more restrictive than
(\ref{l1}).

\section{Acknowledgements}

The author is grateful to T. Bandman, S. Kislyakov, and A. Lerner
for stimulating discussions.

\end{document}